\documentclass{article}
\usepackage[english]{babel}
\usepackage[letterpaper,top=2cm,bottom=2cm,left=3cm,right=3cm,marginparwidth=1.75cm]{geometry}
\usepackage{tgschola}
\usepackage{amsmath}
\usepackage{amsfonts}
\usepackage{amssymb}
\usepackage{amsthm}
\usepackage{mathtools}
\usepackage{asymptote}
\usepackage{changepage}
\usepackage{graphicx}
\usepackage{url}
\usepackage[labelfont=bf]{caption}
\usepackage[colorlinks=true, allcolors=blue]{hyperref}

\newtheorem{theorem}{Theorem}

\newtheorem*{claim}{Claim}

\title{Hinged-rulers fold in $2-\Theta(\frac{1}{2^{n/4}})$}
\author{Colin Tang\footnote{Carnegie Mellon University, United States of America, \url{cstang@andrew.cmu.edu}}}

\begin{document}
\maketitle
\begin{abstract}
    A \textbf{hinged-ruler} is a sequence of line segments in the plane joined end-to-end with hinges, so each hinge joins exactly two segments, the first segment and last segment are adjacent to only one hinge each, and all other segments are adjacent to exactly two hinges. Hopcroft, Joseph, and Whitesides first posed the \textbf{hinged-ruler-folding problem} in \cite{hopcroft85}: given a hinged-ruler and a real number $K$, can the hinged-ruler be folded so as to fit within a one-dimensional interval of length $K$? We show that if the segment lengths are constrained to be real numbers in the interval $[0,1]$, then we can always fold the hinged-ruler so as to fit within a one-dimensional interval of length $2-\Omega(\frac{1}{2^{n/4}})$. On the other hand, we give a construction for a hinged-ruler which cannot be folded into any one-dimensional interval of length smaller than $2-O(\frac{1}{2^{n/4}})$.
\end{abstract}

\section{Introduction and definitions}

Let $n$ be a positive integer. Given elements $e_1,e_2,\cdots, e_n \in \{-1,+1\}$ and nonnegative real numbers $a_1, a_2, \cdots, a_n \in [0,1]$, define variables $s_0, s_1, s_2, \cdots, s_n$ as follows: \begin{align*}
    s_0 &= 0\\
    s_1 &= e_1a_1 \\
    s_2 &= e_1a_1 + e_2a_2 \\
    s_3 &= e_1a_1 + e_2a_2 + e_3a_3\\
    &\vdots \\
    s_n &= e_1a_1 + e_2a_2 + e_3a_3 + \cdots + e_na_n
\end{align*} In other words, $s_i = \sum_{k=1}^i e_ka_k$ for each $0 \le i \le n$. We may interpret the $a_i$ and $e_i$ as follows: A \textbf{hinged-ruler} is a series of line segments of length $a_1,a_2,\cdots, a_n$, joined end-to-end by hinges, so that the first segment (of length $a_1$) and the last segment (of length $a_n$) are adjacent to only one hinge each, but all other segments are adjacent to two hinges each. The hinges are constrained to take on only $0^\circ$ or $180^\circ$ bending angles; i.e. all the line segments must be parallel. Each line segment may point in the \emph{positive} or \emph{negative} directions, according to the value of $e_i$. 

Define the \textbf{range} $g(e_1,e_2,\cdots,e_n,a_1,a_2,\cdots,a_n)$ to be $(\max_k s_k) - (\min_k s_k)$; i.e. the range is the difference between the largest $s_i$ and the smallest $s_i$. In the hinged-ruler interpretation, the range is simply the length of the hinged-ruler once it has been folded in a particular way.

\begin{figure}
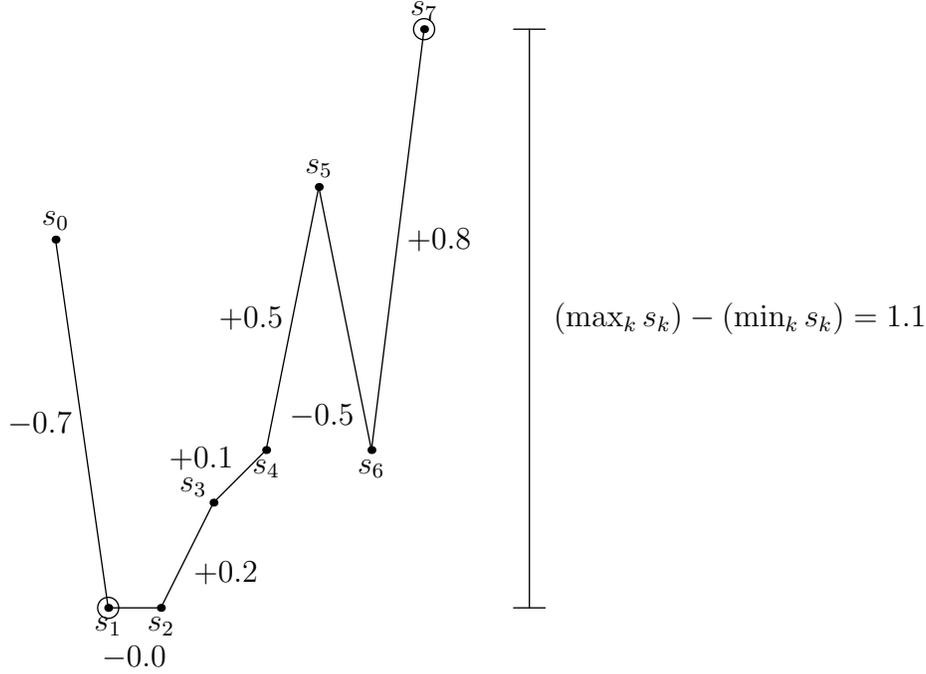

    \centering
\begin{asy}
unitsize(0.7cm);
pair p0 = (0,0); dot(p0); label("$s_0$",p0,N);
pair p1 = (1,-7); dot(p1); label("$s_1$",p1,S); draw(circle(p1,0.2));
pair p2 = (2,-7); dot(p2); label("$s_2$",p2,S);
pair p3 = (3,-5); dot(p3); label("$s_3$",p3,NW);
pair p4 = (4,-4); dot(p4); label("$s_4$",p4,S);
pair p5 = (5,1); dot(p5); label("$s_5$",p5,N);
pair p6 = (6,-4); dot(p6); label("$s_6$",p6,S);
pair p7 = (7,4); dot(p7); label("$s_7$",p7,N); draw(circle(p7,0.2));
draw(p0--p1--p2--p3--p4--p5--p6--p7);
label("$-0.7$",0.5*p0+0.5*p1,W);
label("$-0.0$",0.5*p1+0.5*p2,4*S);
label("$+0.2$",0.5*p2+0.5*p3,SE);
label("$+0.1$",0.5*p3+0.5*p4,NW);
label("$+0.5$",0.5*p4+0.5*p5,W);
label("$-0.5$",0.2*p5+0.8*p6,SW);
label("$+0.8$",0.5*p6+0.5*p7,E);
pair mi = (9,-7);
pair ma = (9,4);
pair rs = (0.3,0);
pair ls = (-0.3,0);
draw(mi--ma); label("$(\max_k s_k) - (\min_k s_k)=1.1$",0.5*mi+0.5*ma+(4,0));
draw(mi+ls--mi+rs);
draw(ma+ls--ma+rs);
\end{asy}
    \caption{The range $g(-,-,+,+,+,-,+,0.7,0.0,0.2,0.1,0.5,0.5,0.8)$ is equal to $1.1$. This is not the optimal choice of signs, as the step-cover $f(0.7,0.0,0.2,0.1,0.5,0.5,0.8)$ is equal to $0.9$.}
    \label{fig:range}
\end{figure}

Now define the \textbf{step-cover} $f(a_1,a_2,\cdots,a_n)$ to be \[f(a_1,a_2,\cdots,a_n)\coloneqq \min_{e_1,e_2,\cdots,e_n \in \{-1,+1\}} g(e_1,e_2,\cdots,e_n,a_1,a_2,\cdots,a_n);\] i.e. the step-cover of $(a_1,a_2,\cdots,a_n)$ is the smallest possible range, as we vary over all choices of signs $e_1,e_2,\cdots, e_n \in \{-1,+1\}$. In the hinged-ruler interpretation, the step-cover is the minimum possible length of the hinged-ruler, as we range over all possible foldings.

Define the \textbf{fit} $p_n$ to be \[p_n \coloneqq \max_{a_1,a_2,\cdots,a_n \in [0,1]} f(a_1,a_2,\cdots, a_n);\] i.e. the fit is the largest real number $M$ such that there exists a choice of $(a_1,a_2,\cdots,a_n)$ such that no matter how we choose the signs $(e_1,e_2,\cdots,e_n)$, the range is forced to be at least $M$. Since we may pad excess range with zeroes, we conclude that $p_n \le p_{n+1}$; i.e. the fit is monotonically nondecreasing in $n$.

Note that $f$ is a continuous piecewise-linear function of $(a_1,a_2,\cdots,a_n)$. More specifically, $f$ is linear on each domain $D_{e_1',e_2',\cdots,e_n',i,j}$, where the domain is defined to be the set of $(a_1,a_2,\cdots,a_n)$ such that $g(e_1,e_2,\cdots,e_n,a_1,a_2,\cdots,a_n)$ is minimized when $(e_1,e_2,\cdots,e_n) = (e_1',e_2',\cdots,e_n')$, and for that choice of $(e_1,e_2,\cdots,e_n)$, we have $\max_k s_k = s_i$ and $\min_k s_k = s_j$. Because each domain may be described as the intersection of half-planes whose defining equations have rational coefficients, and because $f$ itself takes the form of a rational linear combination of the $(a_1,a_2,\cdots,a_n)$ on each domain, we conclude that the fit is a rational number.

In 1985, Hopcroft, Joseph, and Whitesides posed the following question in \cite{hopcroft85}, in the context of robotic motion planning: given $(a_1,a_2,\cdots,a_n)$, what is their step-cover $f(a_1,a_2,\cdots,a_n)$? More specifically, define the decision problem \textsf{Ruler-Folding} as follows: given rationals $(a_1,a_2,\cdots,$ $a_n)$ and $K$, we say that $(a_1,a_2,\cdots,a_n,$ $K)$ $\in$ $\textsf{Ruler-Folding}$ if and only if $f(a_1,a_2,\cdots,a_n) \le K$. The authors prove that the decision problem \textsf{Ruler-Folding} is \textsf{NP}-complete using a reduction from the decision problem \textsf{Partition}, which is known to be \textsf{NP}-complete. The authors also show that the fit $p_n$ is always at most $2$, by using a greedy algorithm. Indeed, given prescribed step sizes $(a_1,a_2,\cdots,a_n)$, we may choose the $(e_1,e_2,\cdots,e_n)$ iteratively such that $|s_i| \le 1$ holds, which implies that all the $s_i$ are in the interval $[-1,1]$ of length $2$. The authors also give an explicit construction showing that $p_n \ge 2 - O(\frac{1}{n})$.

In 2005, C\u{a}linescu and Dumitrescu considered the two-dimensional ruler-folding problem in \cite{calinescu05}, where the $e_i$ are allowed to be arbitrary unit vectors in $\mathbb{R}^2$. In the hinged-ruler interpretation, the limitation that the hinges must follow $0^\circ$ or $180^\circ$ bending angles is removed, and each hinge may take on any bending angle from $0^\circ$ through $360^\circ$. The resulting hinged-ruler will not, in general, be contained in a one-dimensional interval, but we may define its two-dimensional range to be simply the smallest possible diameter of a convex set containing the hinged-ruler. The authors prove that the two-dimensional step-cover is always equal to $\max_k a_k$. In other words, any hinged-ruler may be folded so as to be contained in a convex set of diameter $1$ (and the convex set in particular may be taken to be a Reuleaux triangle). We will not consider the two-dimensional problem in this paper, choosing to focus on the one-dimensional problem.

In this paper we prove the following bounds on the fit, establishing that the fit $p_n$ is $2 - \Theta(\frac{1}{2^{n/4}})$:

\begin{theorem}\label{bound}
If $n = 4m$ is a multiple of $4$, we have $p_n \le 2 - \frac{1}{2^{m+3}-7}$.
\end{theorem}

\begin{theorem}\label{construction}
If $n = 4m-1$ is $1$ less than a multiple of $4$, we have $p_n \ge 2 - \frac{1}{3\cdot 2^{m-1}-1}$.
\end{theorem}

\section{Proofs}

\textit{Proof of Theorem \ref{bound}.} Assume that $n=4m$ is a multiple of $4$, and define $\epsilon = \frac{1}{2^{m+3}-7}$. 

Let $a_1,a_2,\cdots,a_n \in [0,1]$ be arbitrary; we aim to show that we may choose $(e_1,e_2,\cdots,e_n)$ such that the resulting range is at most $2 - \epsilon$. First note that if there exists $i$ such that $a_i + a_{i+1} \le 1$, we may replace $(a_1,a_2,\cdots, a_n)$ with \[(a_1,a_2,\cdots,a_{i-1},a_i+a_{i+1},a_{i+2}, \cdots, a_n, 1).\] Thus, we may assume that there does not exist $i$ such that $a_i + a_{i+1} \le 1$. In other words, for each $i$, we have $a_i + a_{i+1} > 1$.

If $q: [-1,1] \rightarrow [0,+\infty)$ is a probability distribution (for our purposes, this will be a piecewise-constant function with integral $1$), define $\Gamma_-(q,a)$ to be the following measurable function:\begin{itemize}
    \item If $-1 \le x \le 1-2a$, then $(\Gamma_-(q,a))(x) = \frac12 q(x+a)$.
    \item If $1-2a \le x \le 1-a$, then $(\Gamma_-(q,a))(x) = q(x+a)$.
    \item If $1-a \le x \le 1$, then $(\Gamma_-(q,a))(x) = 0$.
\end{itemize} Similarly, we define $\Gamma_+(q,a)$ to be the following measurable function:\begin{itemize}
    \item If $-1+2a \le x \le 1$, then $(\Gamma_+(q,a))(x) = \frac12 q(x-a)$.
    \item If $-1+a \le x \le -1+2a$, then $(\Gamma_+(q,a))(x) = q(x-a)$.
    \item If $-1 \le x \le -1+a$, then $(\Gamma_+(q,a))(x) = 0$.
\end{itemize} Note that $\Gamma_-(q,a)$ and $\Gamma_+(q,a)$ are well-defined almost everywhere. Finally, we define the probability distribution $\Phi(q,a)$ to be simply $\Phi(q,a) \coloneqq \Gamma_-(q,a) + \Gamma_+(q,a)$. It is easily verified that if $\int q \, dx = 1$, then for any $a \in [0,1]$, we have $\int \Phi(q,a) \, dx = 1$, so $\Phi(q,a)$ is indeed a probability distribution.

Intuitively, we may describe $\Phi(q,a)$ as follows. A particle $P$ is initially placed at a location $x$ in $[-1,1]$, where the probability distribution of $x$ is given by $q$. Then: \begin{itemize}
    \item If $-1 \le x \le -1+a$, $P$ moves to the right by a distance of exactly $a$ with probability $1$.
    \item If $-1+a \le x \le 1-a$, $P$ moves left or right by a distance of exactly $a$, with the two directions chosen with probability $\frac12$ each.
    \item If $1-a \le x \le 1$, $P$ moves to the left by a distance of exactly $a$ with probability $1$.
\end{itemize} Then $\Phi(q,a)$ is the probability distribution of the final position $x'$ of $P$.

\begin{figure}
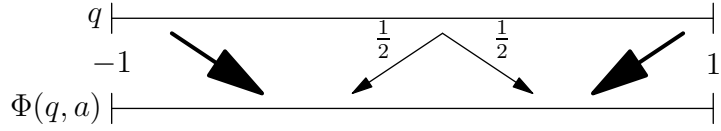

    \centering
\begin{asy}
unitsize(2cm);
pair ut = (0,0.1);
pair dt = (0,-0.1);
pair us = (0,0.3);
pair ds = (0,-0.3);
pair ls = (-2,0);
pair rs = (2,0);
draw(ls+us--rs+us);
draw(ls+us+ut--ls+us+dt);
draw(rs+us+ut--rs+us+dt);
draw(ls+ds--rs+ds);
draw(ls+ds+ut--ls+ds+dt);
draw(rs+ds+ut--rs+ds+dt);
label("$-1$",ls);
label("$1$",rs);
label("$q$",ls+us,W);
label("$\Phi(q,a)$",ls+ds,W);
pen th = linewidth(1.5);
draw((-1.6,0.2)--(-1.0,-0.2), p = th, arrow = Arrow()); 
draw((0.2,0.2)--(-0.4,-0.2), arrow = Arrow()); label("$\frac12$",(-0.1,0),0.8*NW);
draw((0.2,0.2)--(0.8,-0.2), arrow = Arrow()); label("$\frac12$",(0.5,0),0.8*NE);
draw((1.8,0.2)--(1.2,-0.2), p = th, arrow = Arrow());
\end{asy}
    \caption{How to calculate $\Phi(q,a)$ given $q$. In the diagram, $a = 0.3$.}
    \label{fig:Phi}
\end{figure}

For each $0 \le i \le 4m$, define a probability distribution $q_i: [-1,1] \rightarrow [0,+\infty)$ satisfying the following criteria: \begin{itemize}
    \item[(i)] $q_{2m}$ is the uniform distribution on $[-1,1]$, so $q_{2m}(x) = \frac12$ for all $-1 \le x \le 1$.
    \item[(ii)] If $2m \le i \le 4m-1$, $q_{i+1}$ is given by $\Phi(q_i,a_{i+1})$.
    \item[(iii)] If $1 \le i \le 2m$, $q_{i-1}$ is given by $\Phi(q_i,a_i)$.
\end{itemize}

\begin{figure}
    \centering
$q_0 \xleftarrow{a_1} \cdots \xleftarrow{a_{2m-2}} q_{2m-2} \xleftarrow{a_{2m-1}} q_{2m-1} \xleftarrow{a_{2m}} q_{2m} \xrightarrow{a_{2m+1}} q_{2m+1} \xrightarrow{a_{2m+2}} q_{2m+2} \xrightarrow{a_{2m+3}} \cdots \xrightarrow{a_{4m}} q_{4m}$
    \caption{The $q_i$ are computed starting with $q_{2m}$ and moving outwards to both ends, with step sizes given by the $a_i$.}
    \label{fig:q}
\end{figure}

Note that for $i \ge 2m+1$, each $q_i$ is the distribution of the endpoint of a random walk where the particle $P$ is initially distributed uniformly on $[-1,1]$ and then takes steps of size $a_{2m+1},a_{2m+2}$, $\cdots$, $a_i$.

For each $i$, let $r_i$ denote the portion of $q_i$ that is in the fringes; i.e. \[r_i = \int_{-1}^{-1+(\epsilon/2)} q_i \, dx + \int_{1-(\epsilon/2)}^1 q_i \, dx.\] Note that $r_i$ is the probability that the $i$th step of the random walk is outside $[-1+\frac{\epsilon}{2},1+\frac{\epsilon}{2}]$. If we can show that $\sum_i r_i < 1$, then with positive probability, the random walk never leaves $[-1+\frac{\epsilon}{2},1+\frac{\epsilon}{2}]$, resulting in a range of at most $2-\epsilon$, as desired. In fact, it suffices to show that $\sum_i r_i \le 1$, due to compactness and the fact that we may expand the fringes slightly.

We have the following key claim.
\begin{claim}
Let $q:[-1,1] \rightarrow [0,+\infty)$ be a probability distribution and let $b, b' \in [0,1]$ be such that $b + b' > 1$. Define $q' = \Phi(q,b)$ and $q'' = \Phi(q',b')$ (the primes do not mean derivatives). Then \[\max_{-1 \le x \le 1}q''(x) \le 2\max_{-1 \le x \le 1} q(x).\] In other words, $q_{4m}$ is at most $2^m$ as ``concentrated'' as $q_{2m}$.
\end{claim}
\begin{proof}
We may assume that $q$ is the uniform distribution, so $q(x) = \frac12$ for all $-1 \le x \le 1$. We wish to show that $q''(x) \le 1$ for all $-1 \le x \le 1$. This is a tedious case check; details are given below.

Note that for each $x$, we have $(\Gamma_-(q,b))(x) \le \frac12$ and $(\Gamma_+(q,b))(x) \le \frac12$, so \[q'(x) = (\Phi(q,b))(x) = (\Gamma_-(q,b))(x) + (\Gamma_+(q,b))(x) \le 1\] for each $x$. Thus, for each $x$, if $(\Gamma_-(q',b'))(x) = 0$ or $(\Gamma_+(q',b'))(x) = 0$, we would automatically have $q''(x) \le 1$, as desired. It remains to check the case where $(\Gamma_-(q',b'))(x) > 0$ and $(\Gamma_+(q',b'))(x) > 0$; this implies that $-1+b' \le x \le 1-b'$.

First suppose that $b \le b'$. Then note that since we have $-1+b' \le x \le 1-b'$, we also have $(\Gamma_-(q,b))(x-b') = \frac12 q(x+b-b') = \frac14$ and $(\Gamma_+(q,b))(x+b') = \frac12 q(x-b+b') = \frac14$. Note that $b + b' > 1$, so at least one of $(\Gamma_+(q,b))(x-b')$ and $(\Gamma_-(q,b))(x+b')$ is equal to $0$. Furthermore, $(\Gamma_+(q,b))(x-b') \le \frac12$ and $(\Gamma_-(q,b))(x+b') \le \frac12$, so their sum \[(\Gamma_+(q,b))(x-b') + (\Gamma_-(q,b))(x+b') \le \frac12 + 0 = \frac12.\] We conclude that \begin{align*}
    q''(x) &= (\Phi(q',b'))(x) \\
    &= (\Gamma_-(q',b'))(x) + (\Gamma_+(q',b'))(x) \\
    &\le q'(x-b') + q'(x+b') \\
    &= (\Phi(q,b))(x-b') + (\Phi(q,b))(x+b') \\
    &= (\Gamma_-(q,b))(x-b') + (\Gamma_+(q,b))(x-b') + (\Gamma_-(q,b))(x+b') + (\Gamma_+(q,b))(x+b') \\
    &= (\Gamma_-(q,b))(x-b') + (\Gamma_+(q,b))(x+b') + ((\Gamma_+(q,b))(x-b') + (\Gamma_-(q,b))(x+b')) \\
    &\le \frac14 + \frac14 + \frac12 = 1
\end{align*} as desired.

Now suppose that $b' \le b$. Observe that if $1-2b' \le x \le 1-b'$, then we have $1-b' \le x+b' \le 1$, which implies $1-b \le x+b' \le 1$. Thus, $(\Gamma_-(q,b))(x+b') = 0$, which implies \begin{align*}(\Gamma_-(q',b'))(x) &\le q'(x+b') \\ &= (\Phi(q,b))(x+b') \\ &= (\Gamma_-(q,b))(x+b') + (\Gamma_+(q,b))(x+b') \\ &= (\Gamma_+(q,b))(x+b') \\ &\le \frac12.\end{align*} On the other hand, if $b' \le \frac23$ and $-1+b' \le x \le 1-2b'$, then we have $(\Gamma_-(q',b'))(x) = \frac12 q'(x+b')$, which implies \begin{align*}(\Gamma_-(q',b'))(x) &= \frac12 q'(x+b') \\
&= \frac12 (\Phi(q,b))(x+b') \\
&= \frac12 ((\Gamma_-(q,b))(x+b') + (\Gamma_+(q,b))(x+b')) \\
&\le \frac12 \left(\frac12 + \frac12\right)\\
&= \frac12.
\end{align*} In any case, we have $(\Gamma_-(q',b'))(x) \le \frac12$. By reflecting the argument about $0$, similar reasoning implies that $(\Gamma_+(q',b'))(x) \le \frac12$. We conclude that \[q''(x) = (\Phi(q',b'))(x) = (\Gamma_-(q',b'))(x) + (\Gamma_+(q',b'))(x) \le \frac12 + \frac12 = 1,\] as desired.
\end{proof}

\begin{figure}
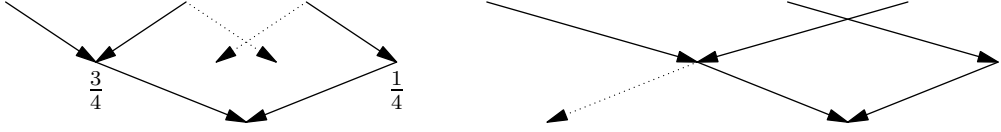

    \centering
\begin{asy}
unitsize(2cm);
pair rs = (4,0);
real s = 2.0;
draw(s*(-0.8,0.2)--s*(-0.5,0),arrow=Arrow());
draw(s*(-0.2,0.2)--s*(-0.5,0),arrow=Arrow());
draw(s*(-0.2,0.2)--s*(0.1,0),arrow=Arrow(),p = dotted);
draw(s*(0.2,0.2)--s*(-0.1,0),arrow=Arrow(), p= dotted);
draw(s*(0.2,0.2)--s*(0.5,0),arrow=Arrow());
draw(s*(-0.5,0)--s*(0,-0.2),arrow=Arrow());
draw(s*(0.5,0)--s*(0,-0.2),arrow=Arrow());

label("$\frac34$",s*(-0.5,0),S);
label("$\frac14$",s*(0.5,0),S);

draw(rs+s*(-1.2,0.2)--rs+s*(-0.5,0),arrow=Arrow());
draw(rs+s*(0.2,0.2)--rs+s*(-0.5,0),arrow=Arrow());
draw(rs+s*(-0.2,0.2)--rs+s*(0.5,0),arrow=Arrow());
draw(rs+s*(-0.5,0)--rs+s*(-1.0,-0.2),arrow=Arrow(),p=dotted);
draw(rs+s*(-0.5,0)--rs+s*(0,-0.2),arrow=Arrow());
draw(rs+s*(0.5,0)--rs+s*(0,-0.2),arrow=Arrow());
\end{asy}
    \caption{A visualization of the proof of Claim. The left half represents the case where $b \le b'$, and the right half represents the case where $b' \le b$.}
    \label{fig:claim}
\end{figure}

Due to Claim, we have that $r_{2m+1} \le 2\epsilon$, $r_{2m+2} \le 2\epsilon$, $r_{2m+3} \le 4\epsilon$, $r_{2m+4} \le 4\epsilon$, etc. Summing yields \begin{align*}\sum_{i=0}^{4m} r_i &\le (2^m + 2^m + \cdots + 4+4+2+2+1+2+2+4+4+\cdots+2^m+2^m)\epsilon \\ &= (1+4(2^1+2^2+\cdots+2^m))\epsilon \\ &= (2^{m+3}-7)\epsilon \\ &= 1,\end{align*} so $\sum_i r_i \le 1$, as desired. $\blacksquare$

\textit{Proof of Theorem \ref{construction}.} Assume that $n = 4m-1$ is $1$ less than a multiple of $4$, and define $\delta = \frac{1}{3\cdot 2^{m-1} - 1}$. We will explicitly construct $(a_1,a_2,\cdots,a_n)$ such that no matter how $(e_1,e_2,\cdots,e_n)$ are chosen, the resulting range is at least $2 - \delta$.

Our construction will be as follows. For each $1 \le i \le n$:\begin{itemize}
    \item If $i$ is odd, then $a_i = 1$.
    \item If $i = 2m$, then $a_i = 1 - 2^{m-1}\delta$.
    \item If $2 \le i \le 2m-2$ and $i$ is even, then $a_i = 1 - 2^{(i/2)-1}\delta$.
    \item If $2m+2 \le i \le 4m-2$ and $i$ is even, then $a_i = 1 - 2^{(2m-1)-(i/2)}\delta$.
\end{itemize}
For example, if $m = 4$ then our $(a_1,a_2,\cdots,a_n)$ will be \[\left(1,\frac{22}{23},1,\frac{21}{23},1,\frac{19}{23},1,\frac{15}{23},1,\frac{19}{23},1,\frac{21}{23},1,\frac{22}{23},1\right).\] 

Assume without loss of generality (by flipping the signs if needed) that $e_{2m} = -1$. If $e_{2m+1} = -1$ also, then we cannot avoid having a range at least $2-\delta$ even if we restrict our attention to only $(a_{2m},a_{2m+1},\cdots,a_n)$. Similarly, if $e_{2m-1} = -1$, then we cannot avoid having a range at least $2-\delta$ even if we restrict our attention to only $(a_1,a_2,\cdots,a_{2m})$. In fact, for any even $i \ge 2m$, if $e_i = e_{i+1}$ then the range is at least $2-\delta$, and for any even $i \le 2m$, if $e_i = e_{i-1}$ then the range is at least $2-\delta$. (This is true because we may start at $a_i$ and work towards the closest endpoint $a_1$ or $a_n$.) It remains to check the case where $-e_1 = e_2$, $-e_3 = e_4$, $\cdots$, $-e_{2m-3} = e_{2m-2}$, $-e_{2m-1} = e_{2m} = -e_{2m+1} = -1$, $e_{2m+2} = -e_{2m+3}$, $\cdots$, $e_{n-3} = -e_{n-2}$, $e_{n-1} = -e_n$.

If we have $(e_{2m-1},e_{2m},e_{2m+1},e_{2m+2}) = (+1,-1,+1,+1)$, then our range is already at least $2 + (2^{m-1}-2^{m-2})\delta$. Similarly, if we have \[(e_{2m-1},e_{2m},e_{2m+1},e_{2m+2},e_{2m+3},e_{2m+4}) = (+1,-1,+1,-1,+1,+1),\] then our range is already at least $2 + (2^{m-1}+2^{m-2}-2^{m-3})\delta$. In general, if the sequence $(e_1,e_2,$ $\cdots,$ $e_n)$ is not perfectly alternating then the range will be at least $2-\delta$. It remains to check the case where the sequence $(e_1,e_2,\cdots,e_n)$ is perfectly alternating; i.e. $e_i = (-1)^{i+1}$ for each $i$. We easily verify that the range in this case is \begin{align*}&1 + (1+2+4+\cdots + 2^{m-2} + 2^{m-1} + 2^{m-2} + \cdots + 4 + 2 + 1)\delta \\ &= 1 + (2^m + 2^{m-1} - 2)\delta \\ &= 2 - \delta, \end{align*} as desired. $\blacksquare$


\bibliographystyle{alpha}
\bibliography{stepcover}





\end{document}